\documentclass[leqno]{article}
\usepackage{geometry}
\usepackage{graphicx}	
\usepackage[cp1251]{inputenc}
\usepackage[english]{babel}
\usepackage{mathtools}
\usepackage{amsfonts,amssymb,mathrsfs,amscd,amsmath,amsthm}

\usepackage{verbatim}

\usepackage{url}

\usepackage{stmaryrd}

\def\thtext#1{
  \catcode`@=11
  \gdef\@thmcountersep{. #1}
  \catcode`@=12
}

\def\threst{
  \catcode`@=11
  \gdef\@thmcountersep{.}
  \catcode`@=12
}

\theoremstyle{plain}

\newtheorem{thm}{Theorem}[section]
\newtheorem{prop}[thm]{Proposition}
\newtheorem{cor}[thm]{Corollary}

\theoremstyle{definition}
\newtheorem{dfn}[thm]{Definition}

\newtheorem{examp}[thm]{Example}

 \catcode`@=11
 \def\.{.\spacefactor\@m}
 \catcode`@=12

\def\R{\mathbb R}

\def\a{\alpha}
\def\b{\beta}
\def\e{\varepsilon}

\def\D{\Delta}
\def\g{\gamma}

\def\l{\lambda}
\def\r{\rho}
\def\s{\sigma}

\def\0{\emptyset}
\def\:{\colon}
\def\<{\langle}
\def\>{\rangle}
\def\[{\llbracket}
\def\]{\rrbracket}

\def\rom#1{\emph{#1}}
\def\({\rom(}
\def\){\rom)}

\def\ss{\subset}

\def\sp{\supset}

\def\x{\times}

\def\bA{{\bar A}}

\def\diam{\operatorname{diam}}

\def\dis{\operatorname{dis}}

\def\dLS{\operatorname{dLS}}

\def\cD{{\cal D}}

\def\cH{{\cal H}}

\def\cP{{\cal P}}
\def\cR{{\cal R}}

\usepackage{epigraph}

\begin{document}
\title{Calculating Gromov--Hausdorff Distance by means of Borsuk Number}
\author{Alexander O.~Ivanov, Alexey A.~Tuzhilin}
\date{}
\maketitle

\begin{abstract}
The purpose of this article is to demonstrate the connection between the properties of the Gromov--Hausdorff distance and the Borsuk conjecture. The Borsuk number of a given bounded metric space $X$ is the infimum of cardinal numbers $n$ such that $X$ can be partitioned into $n$ smaller parts (in the sense of diameter). An exact formula for the Gromov--Hausdorff distance between bounded metric spaces is obtained under the assumption that the diameter and cardinality of one space are less than the diameter and Borsuk number of another, respectively. Using the results of Bacon's equivalence between the Lusternik--Schnirelmann and Borsuk problems, several corollaries are obtained.
\end{abstract}

\setlength{\epigraphrule}{0pt}

\section{Introduction}
\markright{\thesection.~Introduction}
A natural idea to compare subsets of a given metric space or, more generally, to compare different metric spaces using appropriate distances, leads to appearance of so-called hyperspaces,  i.e., metric spaces of some spaces, see for example~\cite{IllNasdler}. For two subsets of a fixed metric space, a natural distance function was defined by F.~Hausdorff~\cite{Hausdorff} as the infimum of positive numbers $r$ such that one subset is contained in the $r$-neighborhood of the other and vice-versa. This function is referred as the \emph{Hausdorff distance}, and it is a metric on the family of all closed bounded subsets of a metric space, see for example~\cite{BurBurIva}. The Hausdorff distance was generalized to the case of two metric spaces by D.~Edwards~\cite{Edwards} and independently by M.~Gromov~\cite{Gromov} by means of isometrical embeddings into all possible metric spaces, see definitions in Section~\ref{sec:GH}. The resulting function is known as the \emph{Gromov--Hausdorff distance\/} between metric spaces.

The geometry of the Gromov--Hausdorff distance is rather tricky and is intensively investigated by many authors, see a review in~\cite{BurBurIva}. The aim of this paper is to demonstrate relations between Gromov--Hausdorff distance properties and the Borsuk Conjecture. Informally speaking, see a short review and references in Section~\ref{sec:BoLuSch} below, the Borsuk Problem for a given bounded metric space $X$ asks to find the least possible number $\b(X)$ of pieces necessary to partition $X$ in such a way that all the parts are smaller than the entire $X$ (in the sense of diameter). The value $\b(X)$ is referred as \emph{Borsuk number}. Recently, see~\cite{ITBor}, the authors have already observed that the Borsuk number can be calculated in terms of the Gromov--Hausdorff distance to a single-distance space of an appropriate diameter and cardinality. In the present paper an exact formula for the Gromov--Hausdorff distance between bounded metric spaces is obtained under the assumptions that the diameter and the cardinality of one space is less than the diameter and the Borsuk number of the other one, respectively, see Theorem~\ref{thm:dist-n-simplex-bigger-dim}. Several corollaries are obtained using P.~Bacon~\cite{Bacon} equivalence results, see Corollaries~\ref{cor:dLS}, \ref{cor:dLSph}, and~\ref{cor:Bor_GH}.

\section{Preliminaries: Gromov--Hausdorff Distance}\label{sec:GH}
\markright{\thesection.~Preliminaries: Gromov--Hausdorff Distance}

Let $X$ be an arbitrary nonempty set. By $\# X$ we denote the cardinality of a set $X$. Recall that a function $\r\:X\x X\to\R$ is called a \emph{metric\/} if it is non-negative, non-degenerate, symmetric, and satisfies the triangle inequality. A set with a metric is called a \emph{metric space}.  If such a function $\r$ is permitted to take infinite values, then we call $\r$ \emph{a generalized metric}. If we omit the non-degeneracy condition, i.e., permit $\r(x,y)=0$ for some distinct $x$ and $y$, then such a function is referred as  \emph{pseudometric}. If $\r$ is non-negative and symmetric only, and $\r(x,x)=0$ for any $x\in X$, then we call such $\r$ by a \emph{distance function}, instead of metric or pseudometric. As a rule, if it is not ambiguous, we write $|xy|$ for $\r(x,y)$.

In what follows, all metric spaces are endowed with the corresponding metric topology. Let $X$ be a metric space. The closure of a subset  $A\ss X$  is denoted by $\bA$. For its arbitrary nonempty subset $A\ss X$ and a point $x\in X$, we put $|xA|=|Ax|=\inf\big\{|ax|: a\in A\big\}$. Further, for $r\ge 0$ we set
$$
B_r(x)=\big\{y\in X: |xy|\le r\big\},\quad U_r(x)=\big\{y\in X: |xy|< r\big\},
$$
and
$$
B_r(A)=\big\{y\in X: |Ay|\le r\big\},\quad U_r(A)=\big\{y\in X: |Ay|< r\big\}.
$$

Recall the basic concepts and results concerning the Hausdorff and Gromov--Hausdorff distances. The details can be found in~\cite{BurBurIva}. For a set $X$, by $\cP_0(X)$ we denote the collection of all nonempty subsets of $X$. Let $X$ be a metric space, and  $A,B\in\cP_0(X)$. It is well-known that the following three expressions
\begin{align*}
& \max\Bigl(\sup\bigl\{|aB|:a\in A\bigr\},\,\sup\bigl\{|Ab|:b\in B\bigr\}\Bigr),\\
& \inf\bigl\{r\in[0,\infty]:A\ss B_r(B)\ \&\ B_r(A)\sp B\bigr\},\\ 
& \inf\bigl\{r\in[0,\infty]:A\ss U_r(B)\ \&\ U_r(A)\sp B\bigr\}
\end{align*}
define the same value which is denoted by $d_H(A,B)$. It is easy to see that $d_H$ is non-negative, symmetric, and $d_H(A,A)=0$ for any nonempty $A\ss X$, thus, $d_H$ is a distance on the family $\cP_0(X)$ of all nonempty subsets of the metric space $X$, moreover, it is a generalized pseudometric on $\cP_0(X)$, i.e., it satisfies the triangle inequality. The function $d_H$ is referred as the \emph{Hausdorff distance}. It is well-known that the Hausdorff distance $d_H$ is a metric on the set $\cH(X)$ of all nonempty closed bounded subsets of the metric space $X$.

Further, let $X$ and $Y$ be metric spaces.
A triple $(X',Y',Z)$, consisting of a metric space $Z$ and its two subsets $X'$ and $Y'$, which are isometric to $X$ and $Y$, respectively, is called the \emph{realization of the pair $(X,Y)$}. The infimum of real numbers $r$ for which there exists a realization $(X',Y',Z)$ of $(X,Y)$ satisfying the inequality $d_H(X',Y')\le r$ is denoted by $d_ {GH}(X,Y)$. The value $d_{GH}(X,Y)$ is evidently non-negative, symmetric, and $d_{GH}(X,X)=0$ for any metric space $X$. Thus, $d_{GH}$ is a generalized distance function on each set of metric spaces.

\begin{dfn}
The value $d_{GH}(X,Y)$ is called \emph{the Gromov--Hausdorff distance\/} between the metric spaces $X$ and $Y$.
\end{dfn}

It is well known that on any set of metric spaces the function $d_{GH}$ is a generalized pseudometric. In the general case, $d_{GH}$ is not a metric, it can take the infinite value and be equal to zero for distinct metric spaces. However, if we restrict ourselves to compact metric spaces, considered up to isometry, then $d_{GH}$ is a metric.

For specific calculations of the Gromov--Hausdorff distance, other equivalent definitions of this distance are useful.

Recall that \emph{a relation\/} between sets $X$ and $Y$ is defined as a subset of the Cartesian product $X\x Y$. Similarly to the case of mappings, for each $\s\in\cP_0(X\x Y)$  there are defined the \emph{image\/} $\s(x):=\bigl\{y\in Y:(x,y)\in\s\bigr\}$ of any $x\in X$ and the \emph{pre-image\/} $\s^{-1}(y)=\bigl\{x\in X:(x,y)\in\s\bigr\}$ of any $y\in Y$. Also, for $A\ss X$ and $B\ss Y$ their \emph{image\/} and \emph{pre-image\/} are defined as the union of the images and pre-images of their elements, respectively. A relation $R$ between $X$ and $Y$ is called a \emph{correspondence\/} if $R(X)=Y$ and $R^{-1}(Y)=X$. Thus, the correspondence can be considered as a surjective multivalued mapping. Denote by $\cR(X,Y)$ the set of all correspondences between $X$ and $Y$.

If $X$ and $Y$ are metric spaces, then for each relation $\s\in\cP_0(X\x Y)$ the value
$$
\dis\s:=\sup\Bigl\{\bigl||xx'|-|yy'|\bigr|:(x,y),\,(x',y')\in\s\Bigr\}
$$
is referred as the \emph{distortion\/} of $\s$.

The key well-known result on the connection between correspondences and the Gromov--Hausdorff distance is the following Theorem.

\begin{thm}\label{thm:GH-metri-and-relations}
For any metric spaces $X$ and $Y$ the equality
$$
d_{GH}(X,Y)=\frac12\inf\bigl\{\dis R:R\in\cR(X,Y)\bigr\}
$$
holds.
\end{thm}

It turns out that it suffice to consider correspondences of a special kind only. For arbitrary nonempty sets $X$ and $Y$, a correspondence $R\in\cR(X,Y)$ is called \emph{irreducible\/} if it is a minimal element of the set $\cR(X,Y)$ with respect to the inclusion. The set of all irreducible correspondences between $X$ and $Y$ is denoted by $\cR^0(X,Y)$.

The following result is evident.

\begin{prop}\label{prop:IrredicubleAndDegrees}
A correspondence $R\in\cR(X,Y)$ is irreducible if and only if for any $(x,y)\in R$ it holds
$$
\min\bigl\{\#R(x),\#R^{-1}(y)\bigr\}=1.
$$
\end{prop}

\begin{thm}\label{thm:irreducinle-exists}
Let $X$, $Y$ be arbitrary nonempty sets. Then for every $R\in\cR(X,Y)$ there exists $R^0\in\cR^0(X,Y)$ such that $R^0\ss R$. In particular, $\cR^0(X,Y)\ne\0$.
\end{thm}

Theorems~\ref{thm:irreducinle-exists} and~\ref{thm:GH-metri-and-relations} imply

\begin{cor}\label{cor:GH-distance-irreducinle}
For any metric spaces $X$ and $Y$ we have
$$
d_{GH}(X,Y)=\frac12\inf\bigl\{\dis R\mid R\in\cR^0(X,Y)\bigr\}.
$$
\end{cor}

We now give another useful description of irreducible correspondences.

\begin{prop}\label{prop:decompose_inrreducible}
For any nonempty sets $X$, $Y$, and each $R\in\cR^0(X,Y)$, there exist and unique partitions $R_X=\{X_i\}_{i\in I}$ and $R_Y=\{Y_i\}_{i\in I}$ of the sets $X$ and $Y$, respectively, such that $R=\cup_{i\in I}X_i\x Y_i$. Moreover, $R_X=\cup_{y\in Y}\bigl\{R^{-1}(y)\bigr\}$, $R_Y:=\cup_{x\in X}\bigl\{R(x)\bigr\}$,
$$
R=\{X_i\x Y_i\}_{i\in I}=\cup_{(x,y)\in R}\{R^{-1}(y)\x R(x)\},
$$
and for each $i$ it holds $\min\{\#X_i,\#Y_i\}=1$.

Conversely, each set $R=\cup_{i\in I}X_i\x Y_i$, where $\{X_i\}_{i\in I}$ and $\{Y_i\}_{i\in I}$ are partitions of nonempty sets $X$ and $Y$, respectively, such that for each $i$ it holds $\min\{\#X_i,\#Y_i\}=1$, is an irreducible correspondence between $X$ and $Y$.
\end{prop}

Let $X$ be an arbitrary set consisting of more than one point, and $m$ a cardinal number, $2\le m\le\# X$. By $\cD_m(X)$ we denote the family of all possible partitions of the set $X$ into $m$ nonempty subsets.

Now let $X$ be a metric space. Then for each $D=\{X_i\}_{i\in I}\in\cD_m(X)$ we put
$$
\diam D=\sup_{i\in I}\diam X_i.
$$
Further, for any nonempty $A,B\ss X$, we put $|AB|=\inf\bigl\{|ab|:(a,b)\in A\x B\bigr\}$, and  $|AB|':=\sup\bigl\{|ab|:(a,b)\in A\x B\bigr\}$. Also notice that $|X_iX_i|=0$, $|X_iX_i|'=\diam X_i$, and hence, $\diam D=\sup_{i\in I}|X_iX_i|'$.

The next result follows easily from the definition of distortion and Proposition~\ref{prop:decompose_inrreducible}.

\begin{prop}\label{prop:disRforPartition}
Let $X$ and $Y$ be arbitrary metric spaces, $D_X=\{X_i\}_{i\in I}$, $D_Y=\{Y_i\}_{i\in I}$, $\#I\ge2$, be some partitions of the spaces $X$ and $Y$, respectively, and $R=\cup_{i\in I}X_i\x Y_i\in\cR(X,Y)$. Then
\begin{multline*}
\dis R=\sup\bigl\{|X_iX_j|'-|Y_iY_j|,\,|Y_iY_j|'-|X_iX_j|: i,j\in I\bigr\}=\\
=\sup\bigl\{\diam D_X,\,\diam D_Y,\,|X_iX_j|'-|Y_iY_j|,\,|Y_iY_j|'-|X_iX_j|: i,j\in I,\, i\ne j\bigr\}.
\end{multline*}
\end{prop}

Here we list several simple cases of exact calculation and estimate of the Gromov--Hausdorff distance.

By $\D_1$ we denote a one-point metric space.

\begin{examp}\label{examp:GH_simple}
For any metric space $X$ we have
$$
d_{GH}(\D_1,X)=\frac12\diam X.
$$
\end{examp}

\begin{examp}\label{examp:dGHbelowEstimate}
Let $X$ and $Y$ be some metric spaces, and the diameter of one of them is finite. Then
$$
d_{GH}(X,Y)\ge\frac12|\diam X-\diam Y|.
$$
\end{examp}

\begin{examp}\label{examp:GH-ineq-max-Diam-X-Y}
Let $X$ and $Y$ be some metric spaces, then
$$
d_{GH}(X,Y)\le\frac12\max\{\diam X,\diam Y\},
$$
in particular, if $X$ and $Y$ are bounded metric spaces, then $d_{GH}(X,Y)<\infty$.
\end{examp}

For an arbitrary metric space $X$ and a real $\l>0$, by $\l X$ we denote the metric space obtained from $X$ by multiplying all distances by $\l$. For $\l=0$ we set $\l X=\D_1$.

\begin{examp}\label{examp:dilatation-of-the-same-X}
For any bounded metric space $X$ and any $\l\ge0$, $\mu\ge0$, we have $d_{GH}(\l X,\mu X)=\frac12|\l-\mu|\diam X$, in particularly, for any $0\le a<b$ the curve $\g(t):=t\,X$, $t\in[a,b]$, is shortest.
\end{examp}

\begin{examp}\label{examp:dilatation-and-GH}
Let $X$ and $Y$ be metric spaces, then for any $\l>0$ we have $d_{GH}(\l X,\l Y)=\l\,d_{GH}(X,Y)$. If, in addition, $d_{GH}(X,Y)<\infty$, then the equality holds for all $\l\ge0$.
\end{examp}

\section{Preliminaries: Borsuk--Lusternik--Schnirelmann Problem}
\markright{\thesection.~Preliminaries: Lusternik--Schnirelmann Problem}\label{sec:BoLuSch}
Let $S^n$ be the standard sphere in $\R^{n+1}$. In~1930, L.~Lusternik and L.~Schnirelmann~\cite{LustSch} proved that every closed cover of $S^n$ by means of $(n+1)$ sets $C_1,\ldots,C_{n+1}$ contains $C_i$ with opposite points $x$ and $-x$, i.e. the sphere $S^n$ cannot be partitioned into $m\le(n+1)$ subsets of smaller diameter. A bit later, in 1933, a Polish mathematician K.~Borsuk asked the following general question: Into how many parts should an arbitrary subset of Euclidean space be partitioned in order to obtain pieces of a smaller diameter? He put forward the following famous conjecture: Any bounded non-one-point subset of $\R^n$ can be partitioned into at most $n+1$ subsets, each of which has a smaller diameter than the original subset. K.~Borsuk himself proved this statement for $n=2$ and for a ball and sphere in $3$-dimensional space, \cite{BorCong} and~\cite{Borsuk}. Next, the conjecture was proved by J.~Perkal (1947) and independently by H.\,G.~Eggleston (1955) for $n=3$, then in 1946 by H.~Hadwiger~\cite{Hadw}, \cite{Hadw2} for convex subsets with smooth boundaries, then for centrally symmetric bodies by A.\,S.~Riesling (1971), and after that almost everyone believed that this was true. However, in 1993 the conjecture was suddenly disproved in general case  by J.~Kahn, and G.~Kalai, see~\cite{KahnKalai}. They constructed a counterexample in dimension $n=1325$, and also proved that the conjecture is not valid for all $n>2014$.  This estimate was consistently improved by Raigorodskii, $n\ge561$, Hinrichs and Richter, $n\ge298$, Bondarenko, $n\ge 65$, and Jenrich, $n\ge 64$, see details in a review~\cite{Raig}. Notice that all the examples are finite subsets of the corresponding spaces, and the best known results of Bondarenko~\cite{Bond} and Jenrich~\cite{Jenr} are the $2$-distance subsets of the unit sphere.

In~\cite{ITBor} the following generalized problem is suggested. Let $X$ be a bounded metric space, and $D=\{X_i\}_{i\in I}$ a partition of $X$. We say that  $D$ is a partition of $X$ into subsets having \emph{strictly smaller diameters}, if there exists $\e>0$ such that $\diam X_i\le\diam X-\e$ for all $i\in I$.

The \emph{Generalized Borsuk Problem\/}: Let $X$ be a bounded metric space and $m$ be a cardinal number such that $2\le m\le\#X$. Does there exist a partition of cardinality $m$ of the space $X$ into subsets of strictly smaller diameter?

It turns out, see~\cite{ITBor}, that the solution to this problem can be given in terms of the Gromov--Hausdorff distance. Denote by $\D_m$ a single-distance metric space of cardinality $m$, all nonzero distances of which are equal to $1$.

\begin{thm}\label{thm:Borsuk}
Let $X$ be an arbitrary bounded metric space, $m$ a cardinal number such that $2\le m\le\#X$, and $\l$, $0<\l<\diam X$,  a real number. The space $X$ can be partitioned into $m$ subsets of strictly smaller diameter if and only if $2d_{GH}(\l\D_m,X)<\diam X$. Moreover, if there is no such partition, then $2d_{GH}(\l\D_m,X)=\diam X$.
\end{thm}

For a bounded metric space $X$, $\#X\ge 2$, denote by $\b(X)$ the infimum of cardinal numbers $n$ such that $X$ can be partitioned into $n$ subsets of strictly smaller diameter. We will call $\b(X)$ the \emph{Borsuk number of $X$}. It is clear that $\b(X)\le\#X$.

\section{Gromov--Hausdorff Distance and Borsuk Number}
\markright{\thesection.~Gromov--Hausdorff Distance and Borsuk Number}
It turns out that the Borsuk number can sometimes help calculate the Gromov--Hausdorff distance.

\begin{thm}\label{thm:dist-n-simplex-bigger-dim}
Let $X$ and $Y$ be bounded metric spaces. Suppose that $\#X<\b(Y)$ and $\diam X\le\diam Y$. Then
$$
2d_{GH}(X,Y)=\diam Y.
$$
\end{thm}

\begin{proof}
Let $R\in\cR^0(X,Y)$ be an arbitrary irreducible correspondence. Consider the corresponding partition $R_Y=\big\{R(x)\big\}_{x\in X}$. Since $\#R_Y=\#X<\b(Y)$, then $X$ can not be partitioned into $\#X$ parts of smaller diameter, and hence $\diam R_Y=\diam Y$. Therefore, by Proposition~\ref{prop:disRforPartition} we have $2d_{GH}(X,Y)\ge \diam Y$. On the other hand,
$$
2d_{GH}(X,Y)\le\max\{\diam X,\diam Y\}=\diam Y,
$$
see Example~\ref{examp:GH-ineq-max-Diam-X-Y}. Theorem is proved.
\end{proof}

Let us pass to examples. The simplest one is a single-distance space $\l\D_m$ of cardinality $m$, whose Borsuk number equals $m$. The following equality had been proved in~\cite{GrigIvaTuzSimpDist}.

\begin{cor}\label{cor:simplex}
Let $X$ be a bounded metric space. Then for any real $\l\ge \diam X$ and any cardinal number $m>\#X$ the equality $d_{GH}(X,\l\D_m)=\l$ holds.
\end{cor}

The next class of examples can be obtained from Borsuk--Ulam and Lusternik--Schnirelmann Theorems. Borsuk proved his conjecture in the case of a sphere $S^2$ on the basis of the Borsuk--Ulam theorem, which states that for any continuous mapping $f\:S^2\to\R^2$ there exists $x\in S^2$ such that $f( x)=f(-x)$. Later, P.~Bacon~\cite{Bacon} showed that these two results are equivalent for a fairly wide class of spaces, namely, the following assertion is true.

\begin{thm}[Bacon]\label{thm:Bacon}
Let $X$ be a normal topological space with free continuous involution $A\:X\to X$, and $n$ be a positive integer. The following statements are equivalent.
\begin{itemize}
\item For any continuous mapping $f\:X\to\R^n$, there exists $x\in X$ such that $f\big(A(x)\big)=f(x)$.
\item For any covering $\{C_1,\ldots,C_{n+1}\}$ of the space $X$ by closed sets, there is at least one containing both $x$ and $A(x)$.
\end{itemize}
\end{thm}

But, as we have already mentioned, the first version of the Borsuk--Lusternik--Schnirelmann theorem worked with the sphere $S^n$. In this case, the involution $A\:x\mapsto -x$ has an important additional property, namely, $\big|xA(x)\big|=\diam S^n$ for any $x\in S^n$. We call such involution \emph{diametrical}.

Let $X$ be a bounded metric space such that for any its closed covering $\{C_i\}$ of cardinality at most $n$ there exists at least one set $C_i$ containing diametrical points $x$ and $x'$, i.e\. points such that $|xx'|=\diam X$. We call such spaces $\dLS_n$-spaces (\emph{diametrical $n$-spaces of Lusternik--Schnirelmann\/}). In particular, if $A\:X\to X$ is a free continuous diametrical involution of a bounded metric space $X$ and for any continuous mapping $f\:X\to\R^n$ there exists $x \in X$ such that $f\big(A(x)\big)=f(x)$, then $X$ is a $\dLS_n$-space by Bacon's theorem. A review on geometry of Lusternik--Schnirelmann spaces can be found in~\cite{Musin}. Note that if $X$ is a $\dLS_n$-space, then $\b(X)>n$.

\begin{cor}\label{cor:dLS}
Let $Y$ be a $\dLS_n$-space for some positive integer $n$, and $X$ be a finite metric space such that $\#X\le n$ and $\diam X\le \diam Y$. Then $2d_{GH}(X,Y)=\diam Y$.
\end{cor}

Let $S^n(R)$ be the standard $n$-dimensional sphere of radius $R$ endowed with intrinsic metric.

\begin{cor}\label{cor:dLSph}
Let $X$ a finite metric space. Then $d_{GH}\big(X,S^n(R)\big)=R$ for any $R\ge(\diam X)/\pi$ and any integer $n$ such that $n+1\ge\#X$.
\end{cor}

Applying Theorem~\ref{thm:Borsuk}, we obtain the following result.

\begin{cor}\label{cor:Bor_GH}
Let $X$ and $Y$ be bounded metric spaces, and $\diam X\le\diam Y$. Assume that for some real $\l$, $0<\l<\diam Y$ and some cardinal $m$, $\#X<m\le\#Y$, we have $2d_{GH}(\l\D_m,Y)=\diam Y$. Then $2d_{GH}(X,Y)=\diam Y$.
\end{cor}

\markright{References}


\begin{thebibliography}{20}

\bibitem{IllNasdler} A.~Illanes, S.~Nadler, \emph{Hyperspaces. Fundamentals and Recent Advances}, Marcel Dekker Inc., New York, Basel, 1999.

\bibitem{Hausdorff}
F.~Hausdorff,  {\it Grundz\"uge der Mengenlehre},  Veit, Leipzig, 1914 [reprinted by Chelsea in 1949].

\bibitem{BurBurIva} D.~Burago, Yu.~Burago, S.~Ivanov, \emph{A Course in Metric Geometry}, Graduate Studies in Mathematics {\bf 33} A.M.S., Providence, RI, 2001.

\bibitem{Edwards}
D.~Edwards, ``The Structure of Superspace'', in {\it Studies in Topology}, Academic Press, 1975.

\bibitem{Gromov}
M.~Gromov,
``Groups of Polynomial Growth and Expanding Maps'',
Publications Mathematiques I.H.E.S., {\bf 53}, 1981.

\bibitem{ITBor}
A.\,O.~Ivanov, A.\,A.~Tuzhilin, ``Gromov--Hausdorff Distances to Simplexes and Some Applications to Discrete Optimisation'', Chebyshev. Sbornik,  {\bf 21} (2), pp.~69--189, 2020.

\bibitem{Bacon} P.~Bacon, ``Equivalent Formulations of the Borsuk--Ulam Theorem'', Canad. J. Math., {\bf 18}
 pp.~492--502, 1966.

\bibitem{LustSch} L.\,A.~Lusternik,  L.\,G.~Schnirelmann,  \emph{Topological Methods in Variational Problems}. Issled. Inst. Matem. i Mekh. pri I MGU, Moscow, 1930 [in Russian].

\bibitem{BorCong}
K.~Borsuk, ``\"Uber die Zerlegung einer $n$-dimensionalen Vollkugel in $n$-Mengen''. In: {\emph Verh. International Math. Kongress Z\"urich},  p.~192, 1932.

\bibitem{Borsuk} K.~Borsuk, ``Drei S\"atze \"uber die $n$-dimensionale euklidische Sph\"are'', Fundamenta Math., {\bf 20}, 177--190, 1933.

\bibitem{Hadw} H.~Hadwiger, ``\"Uberdeckung einer Menge durch Mengen kleineren Durchmessers'', Commentarii Mathematici Helvetici, {\bf 18} (1), 73--75, 1945.

\bibitem{Hadw2} H.~Hadwiger, ``Mitteilung betreffend meine Note: \"Uberdeckung einer Menge durch Mengen kleineren Durchmessers'', Commentarii Mathematici Helvetici, {\bf 19}, 1946.

\bibitem{KahnKalai} J.~Kahn, G.~Kalai, \emph{A Counterexample to Borsuk’s Conjecture}. Bull. Amer. Math. Soc., {\bf 29} (1), 60--62, 1993.

\bibitem{Raig} A.\,M.~Raigorodskii, ``Around Borsuk’s Hypothesis'',  Journal of Mathematical Sciences, {\bf 154} (4),  604--623, 2008.

\bibitem{Bond} A.\,V.~Bondarenko,  ``On Borsuk’s Conjecture for Two-Distance Sets'', arXiv e-prints, \texttt{arXiv:1305.2584}, 2013.

\bibitem{Jenr} T.~Jenrich,  ``A $64$-dimensional Two-Distance Counterexample to Borsuk's Conjecture'', arXiv e-prints, \texttt{arXiv:1308.0206}, 2013.

\bibitem{GrigIvaTuzSimpDist}
D.\,S.~Grigor'ev, A.\,O.~Ivanov, A.\,A.~Tuzhilin, ``Gromov--Hausdorff Distance to Simplexes'', Chebyshev. Sbornik, {\bf 20} (2), 100--114, 2019 [in Russian], see also ArXiv e-prints, {\tt arXiv:1906.09644}, 2019.

\bibitem{Musin}
O.\,R.~Musin, A.\,Yu.~Volovikov, ``Borsuk--Ulam Type Spaces'',
ArXiv e-prints, {\tt arXiv: 1507.08872}, 2015.

\end{thebibliography}
\end{document}